\documentclass[10pt]{article}
\usepackage{graphicx}
\usepackage{tikz}
\usepackage[margin=1in]{geometry}
\usepackage{amsmath, amssymb, amsfonts, amsthm}
\newtheorem{lem}{\noindent {\bf Lemma}}[section]

\newtheorem{coro}{\noindent {\bf Corollary}}[section]
\newtheorem{thm}{\noindent {\bf Theorem}}[section]

\newcounter{remark}
\newenvironment{remark}{\smallskip\noindent {\bf Remark \arabic{section}.\arabic{remark}.}}
{\addtocounter{remark}{1}\par}
\newcounter{example}
\newenvironment{exa}{\smallskip\noindent {\bf Example \arabic{section}.\arabic{example}.}}
{\addtocounter{example}{1}\par}
\catcode`@=11 
\date{}
\newcounter{defi}
\newenvironment{defi}{
\smallskip \noindent
{\bf
  Definition \arabic{section}.\arabic{defi}.
}}{\addtocounter{defi}{1}\par}
\newcommand{\ZZ}{\mathbb Z}
\newcommand{\NN}{\mathbb N}
\newcommand{\RR}{\mathbb R}

\newcommand{\abs}[1]{\left|~#1~\right|}
\title{\bf Metric spaces with complexity of the smallest infinite ordinal number}

\author{\large  Jingming Zhu,~~~Yan Wu
\footnote{
College of Mathematics Physics and Information Engineering, Jiaxing University, Jiaxing , 314001, P.R.China.
 E-mail: yanwu@mail.zjxu.edu.cn;~~ 122411741@qq.com}}
\date{}

\begin{document}
\maketitle
\begin{center}
\begin{minipage}{0.9\textwidth}
\noindent{\bf Abstract.} In this paper, we are concerned with
 the study on metric spaces with complexity of the smallest infinite ordinal number. We
 give equivalent formulations of the definition of metric spaces with complexity of the smallest infinite ordinal number and prove that the exact complexity of the finite product $\ZZ \wr\ZZ\times\ZZ \wr\ZZ\times\cdots\times \ZZ \wr\ZZ$ of wreath product is $\omega$, where $\omega$ is the smallest infinite ordinal number. Consequently, we obtain that
the complexity of $(\ZZ \wr\ZZ)\wr\ZZ$ is $\omega+1$.


\end{minipage}
\end{center}
\footnote{
This research was supported by
the National Natural Science Foundation of China under Grant (No.11301224,11401256,11501249)

}
\begin{section}{Introduction}

Inspired by the property of finite asymptotic dimension of M.Gromov~(\cite{Gromov}),
a geometric concept of finite decomposition complexity was
 introduced by E.Guentner, R.Tessera and G.Yu. Roughly
speaking, a metric space has finite decomposition complexity when
there is an algorithm to decompose the space into nice pieces in
certain asymptotic way. It turned out that many groups have finite
decomposition complexity and these groups satisfy strong rigidity
properties including the stable Borel conjecture~(\cite{Yu2012},\cite{Yu2013}).  In \cite{Yu2013},
E.Guentner, R.Tessera and G.Yu show that the class of groups with
finite decomposition complexity includes all linear groups,
subgroups of almost connected Lie groups, hyperbolic groups and
elementary amenable groups and is closed under taking subgroups, extensions, free
amalgamated products, HNN-extensions and inductive limits.

Finite decomposition complexity is a large scale property of a metric space.
To make the
property quantitative, a countable ordinal ¡±the complexity¡± can be defined for a metric space
with finite decomposition complexity. There is a sequence of subgroups of Thompson's group F which is defined by induction as follows:
\[
G_{1}=\ZZ\wr\ZZ,~~G_{n+1}=G_{n}\wr\ZZ.
\]
We are concerned with the study of the exact complexity of $G_{n}$ which is partially inspired by the question of
the finite decomposition complexity of Thompson's group F(\cite{cannon1996},\cite{Yan2014},\cite{Yan2011}).
In fact, if the exact complexity of the sequence $\{G_{n}\}$ of subgroups of F is strictly increasing, then we can prove that Thompson's group F does not
have finite decomposition complexity.
In \cite{Yan2011}, we proved that the complexity of $G_{n}$ is $\omega n$ and the exact complexity of $\ZZ\wr\ZZ$ is $\omega$, where $\omega$ is the smallest infinite ordinal number, but it is still unknown about the exact complexity of $G_{n}$ when $n>1$.
Here we prove that the exact complexity of the finite product $\ZZ \wr\ZZ\times\ZZ \wr\ZZ\times\cdots\times \ZZ \wr\ZZ$ of wreath product is $\omega$.  Consequently, we obtain that
the complexity of $(\ZZ \wr\ZZ)\wr\ZZ$ is $\omega+1$.

There is no group of examples known which make a difference between the exact complexity of $\omega$ and the exact complexity of $\alpha$, where $\alpha$ is a countable ordinal greater than $\omega$. So the question arises naturally: Is there any metric space with the exact complexity greater than $\omega$?
Here we give equivalent descriptions of the definition of metric spaces with complexity of $\omega$.

\end{section}

\begin{section}{Equivalent descriptions of metric spaces with complexity of $\omega$}\

We begin by recalling some elementary concepts from coarse geometry.

Let~$(X, d)$ be a metric space. For $U,V\subseteq X$, let
\[
\text{diam}~ U=\text{sup}\{d(x,y): x,y\in U\}
\]
and
\[
d(U,V)=\text{inf}\{d(x,y): x\in U,y\in V\}.
\]

A family $\mathcal{U}$ of subsets of $X$ is said to be \emph{uniformly bounded} if
$
diam~\mathcal{U}\stackrel{\bigtriangleup}{=}\text{sup}\{\text{diam}~ U: U\in \mathcal{U}\}
$ is finite.

A family $\mathcal{U}$ of subsets of $X$ is said to be\emph{ $r$-disjoint} if
\[
d(U,V)\geq r~~~~~\text{for every}~ U\neq V\in \mathcal{U}.
\]

\begin{defi}(\cite{Bell2011})
A metric space $X$ has \emph{finite asymptotic dimension} if
there is a $n\in\NN$, such that for every $r>0$,
there exists a sequence of uniformly bounded families
$\{\mathcal{U}_{i}\}_{i=1}^{n}$ of subsets of $X$
such that the union
$\bigcup_{i=1}^{n}\mathcal{U}_{i}$ covers $X$ and each $\mathcal{U}_{i}$
is $r$-disjoint.

\end{defi}

 Let  $\mathcal{X}$ and $\mathcal{Y}$ be metric families.
A map of families from $\mathcal{X}$ to $\mathcal{Y}$ is a collection of functions
$F = \{f \}$, each mapping some $X \in \mathcal{X}$ to some $Y \in \mathcal{Y}$ and such that every $X \in \mathcal{X}$ is the
domain of at least one $f \in F$. We use the notation $F : \mathcal{X}\rightarrow\mathcal{Y}$ and, when confusion could
occur, write $f : X_{f }\rightarrow Y_{f }$ to refer to an individual function in $F$.

\begin{defi}
A map of families $F : \mathcal{X}\rightarrow\mathcal{Y}$ is \emph{uniformly expansive} if there exists a non-decreasing
function $\theta : [0,\infty)\rightarrow[0,\infty)$ such that for every $f \in F$ and every $x, y \in X_{f}$,
\[
d(f(x),f(y))\leq \theta(d(x,y)).
\]
$F : \mathcal{X}\rightarrow\mathcal{Y}$ is \emph{effectively proper} if there exists a proper non-decreasing function
$\delta : [0,\infty)\rightarrow[0,\infty)$ such that for every $f \in F$ and every $x, y \in X_{f}$,
\[
d(f(x),f(y))\geq \delta(d(x,y)).
\]
And $F : \mathcal{X}\rightarrow\mathcal{Y}$ is a \emph{coarse embedding} if it is both uniformly expansive and effectively proper.

Two maps $f,g:X\longrightarrow Y$ are \emph{close} if $\{d(f(x), g(x)): x\in X\}$ is a bounded set.
If $f:X\longrightarrow Y$ is a coarse embedding and there exists a coarse embedding $g:Y\longrightarrow X$
such that $f\circ g$ and $g\circ f$ are close to the identities on $X$ and $Y$ respectively, then $f$ is called a \emph{coarse equivalence}.

\end{defi}

\begin{defi}{(\cite{Yu2012},\cite{Yu2013})}
\label{def:rdecomp}
A metric family $\mathcal{X}$ is \emph{$r$-decomposable} over a
metric family $\mathcal{Y}$ if every $X\in \mathcal{X}$ admits a
decomposition \[ X =X_{0}\cup X_{1},
X_{i}=\bigsqcup_{r-\text{disjoint}} X_{ij},
\]
where each $X_{ij}\in \mathcal{Y}$. It is denoted by
$\mathcal{X}\stackrel{r}{\rightarrow}\mathcal{Y}$.
\end{defi}
\begin{remark}
To express the idea that $X$ is the union of $X_{i}$ and the collection of these subspaces $\{X_{i}\}$ is $r$-disjoint,
we write
\[
X=\bigsqcup_{r-\text{disjoint}} X_{i}.
\]
\end{remark}
\begin{defi}{(\cite{Yu2012},\cite{Yu2013})}
\label{def:bound}
\begin{itemize}
\item[(1)]Let $\mathcal{D}_{0}$ be the collection of uniformly bounded families:
$\mathcal{D}_{0}=\{\mathcal{X}: \mathcal{X} \text{ is uniformly bounded
}\}$.
\item[(2)]Let $\alpha$ be an ordinal greater
than 0, let $\mathcal{D}_{\alpha}$ be the collection of metric
families decomposable over $\displaystyle\bigcup_{\beta<\alpha}\mathcal{D}_{\beta}$:
\[\mathcal{D}_{\alpha}=\{\mathcal{X}: \forall~ r>0, ~\exists~
\beta<\alpha, ~\exists~ \mathcal{Y}\in\mathcal{D}_{\beta}, ~\text{
such that } \mathcal{X}\stackrel{r}{\rightarrow}\mathcal{Y} \}.\]
\end{itemize}
\end{defi}

\begin{defi}{\rm(\cite{Yu2012},\cite{Yu2013})}
\begin{itemize}
\item A metric family $\mathcal{X}$ has \emph{finite decomposition
complexity}  if there exists a countable ordinal $\alpha$ such that
$\mathcal{X}\in\mathcal{D_{\alpha}}$.

\item  We say that \emph{the complexity} of the metric family $\mathcal{X}$ is $\alpha$ if~$
\mathcal{X}\in \mathcal{D}_{\alpha}.$
\item  We say that  \emph{the exact complexity} of $\mathcal{X}$ is $\alpha$ if~$
\mathcal{X}\in \mathcal{D}_{\alpha}$ and $\forall~\beta<\alpha$, $\mathcal{X}~\bar{\in}~\mathcal{D}_{\beta}.$

\end{itemize}
\end{defi}

\begin{remark}
\begin{itemize}
\item Note that for any $ \beta<\alpha,
\mathcal{D}_{\beta}\subseteq\mathcal{D}_{\alpha}.$
\item  We view a single metric space $X$ as a
metric family with a single element.
\end{itemize}
\end{remark}

~It is known that a metric space $X$ has finite asymptotic dimension if and only if $X\in\mathcal{D}_{n}$ for some $n\in\NN$~ (\cite{Yu2013},\cite{Yan2011}).\\

\begin{defi}
~We say that a metric family $\mathcal{X}$ has \emph{uniformly finite asymptotic dimension} if~$
\mathcal{X}\in\mathcal{D}_{n}$ for some $n\in\NN$.
\end{defi}

\begin{lem}{\rm(\cite{Yu2012},\cite{Yu2013})}(Coarse invariance)
\label{lem:qisomecomplexity}
Let $\mathcal{X}$ and $\mathcal{Y}$ be two metric families and there is a coarse embedding~~$
\phi: \mathcal{X}\rightarrow \mathcal{Y}$. If $\mathcal{Y}\in\mathcal{D_{\alpha}}$~~for some countable ordinal $\alpha$, then $\mathcal{X}\in\mathcal{D_{\alpha}}$. Consequently, if $\phi$ is a coarse equivalence, then $\mathcal{X}\in\mathcal{D_{\alpha}}$~~if and only if
$\mathcal{Y}\in\mathcal{D_{\alpha}}$.
In particular, if $X$ is a subspace of a metric space $Y$ and $Y\in\mathcal{D_{\alpha}}$, then $X\in\mathcal{D_{\alpha}}$.
\end{lem}

It is easy to obtain the following two Lemmas by simple induction.
\begin{lem}{\rm(\cite{Yan2011})}
\label{lem:inv}
Let $X$ be a metric space with a left-invariant metric and
$\{X_{i}\}_{i}$ be a sequence of subspaces of $X$ with the induced
metric. $\text{If }\{X_{i}\}_{i}\in\mathcal{D}_{\alpha}, \text{then
}\{gX_{i}\}_{g,i}\in\mathcal{D}_{\alpha}, $ $\text{where } gX_{i}=\{
gh |h\in X_{i} \}$.

\end{lem}

\begin{lem}{\rm(\cite{Yan2011})}
\label{lem:metric}
Let $\mathcal{X}=\{X_{i}\}$ and  $\mathcal{Y}=\{Y_{j}\}$ be metric
families, $\mathcal{X\times
Y}\stackrel{\bigtriangleup}{=}\{X_{i}\times Y_{j}\}$. Then for any $ m, n\in \NN,$
if $\mathcal{X}\in \mathcal{D}_{ m}$ and
$\mathcal{Y}\in \mathcal{D}_{n}$, then
$\mathcal{X}\times \mathcal{Y}\in\mathcal{D}_{m+n}$.

\end{lem}

\begin{exa}
\label{ex:ZZ}
$\text{Let}~~ G=\bigoplus_{i=1}^{\infty}\ZZ\text{~~~(countable infinite direct sum)}$ with
 \[
d_{1}(g,h)=\sum_{n=1}^{\infty}\abs{n}\abs{g_{n}-f_{n}}, \forall
g=(g_{1},\cdots,g_{n},\cdots), h=(h_{1},\cdots,h_{n},\cdots)\in G.\]
It was proved that $(G, d_{1})\in\mathcal{D}_{\omega}$ , where $\omega$ is the smallest infinite ordinal number
~(\cite{Yu2013}).
\end{exa}

Inspired by the equivalent descriptions of finite asymptotic dimension ~(\cite{Bell2011}), here we give the equivalent descriptions of
metric spaces with complexity $\omega$.

Before stating the theorem, we recall some necessary definitions.

Let $X$ be a metric space and let $\mathcal{U}$ be a cover of $X$, the \emph{Lebesgue number} $L(\mathcal{U})$ of $\mathcal{U}$ is the largest number
~$\lambda$~such that if $A\subseteq X$ and diam~$A\leq\lambda$, then there exists some $U\in\mathcal{U}$ such that $A\subseteq U.$ \emph{The multiplicity}
$m(\mathcal{U})$ of $\mathcal{U}$ is the maximal number of elements of  $\mathcal{U}$ with a nonempty intersection. \emph{The $d$-multiplicity} of  $\mathcal{U}$ is
defined to be the largest $n$ such that there is a $x\in X$, satisfying $B_{d}(x)$ meets $n$ sets in $\mathcal{U}$.
A map $\varphi: X\rightarrow Y$ between metric spaces is \emph{$\varepsilon$ Lipschitz} if
\[
d(\varphi(x_{1}),\varphi(x_{2}))\leq \varepsilon d(x_{1},x_{2})~~\text{for every}~x_{1}\neq x_{2}\in X.
\]
We use the notation $l_{2}$ for the Hilbert space of square summable sequences, i.e.
\[
l_{2}=\{(x_{1}, x_{2},\cdots)~~|~~\sum_{n=1}^{\infty}\abs{x_{n}}^{2}<\infty\}.
\]
Let $\Delta$ denote the standard infinite dimensional simplex in $l_{2}$, i.e.
\[
\Delta=\{(x_{1}, x_{2},\cdots)\in l_{2}~~|~~\sum_{n=1}^{\infty}x_{n}=1, ~x_{n}\geq0\}.
\]
\emph{A uniform complex} is a simplicial complex considered to be a subset of $\Delta$ with each vertex at some basis element with the restricted metric.

\begin{thm}

Let $X$ be a metric space. The following conditions are equivalent.

\begin{itemize}
 \item[(1)] $X\in  \mathcal{D}_{\omega}$. i.e.  for every $r>0$, there exist $r$-disjoint families $\mathcal{V}_{0}$ and $\mathcal{V}_{1}$
 such that $\mathcal{V}_{0}\cup\mathcal{V}_{1}$ covers $X$ and the family $\mathcal{V}_{0}\cup\mathcal{V}_{1}$ has uniformly finite asymptotic dimension.
 \item [(2)]
 For every $d>0$, there exists a uniformly finite asymptotic dimension cover $\mathcal{V}$ of $X$ with $d$-multiplicity $\leq$ 2. i.e.
 \[
 \forall~x\in X, ~~~\sharp\{V\in \mathcal{V}|V\cap B_{d}(x)\neq \varnothing \}\leq2.
 \]
 \item [(3)]
 For every $\lambda>0$, there exists a uniformly finite asymptotic dimension cover $\mathcal{W}$ of $X$ with the Lebesgue number $L(\mathcal{W})>\lambda$
and the multiplicity
$m(\mathcal{W})\leq 2$.
 \item [(4)]
 For every $\varepsilon>0$, there exists an $\varepsilon$-Lipschitz map $\varphi: X\rightarrow K\subseteq\Delta$ to a uniform simplicial complex of dimension 1 such that $\{\varphi^{-1}(\text{st}~v)|v\in K^{(0)}\}$ has uniformly finite asymptotic dimension, where $\text{st}~v$ is the star of the vertex $v$ in the complex $K$.
 \item [(5)] For every uniformly bounded cover $\mathcal{V}$ of $X$, there is a cover $\mathcal{U}$ of $X$ with uniformly finite asymptotic dimension such that $\mathcal{V}$ refines $\mathcal{U}$ (i.e. every $V\in\mathcal{V}$ is contained in some element $U\in\mathcal{U}$)~and the multiplicity
$m(\mathcal{U})\leq 2$.
\end{itemize}

\begin{itemize}
\item (1)$\Rightarrow$(2):
For every $d>0$ and $r>2d$, there exist $r$-disjoint families $\mathcal{V}_{0}$ and $\mathcal{V}_{1}$
 such that $\mathcal{V}_{0}\cup\mathcal{V}_{1}$ covers $X$ and the family $\mathcal{V}_{0}\cup\mathcal{V}_{1}$ has uniformly finite asymptotic dimension.
 Let $\mathcal{V}=\mathcal{V}_{0}\cup\mathcal{V}_{1}$. For $x\in X$, if $V_{1}\cap B_{d}(x)\neq \varnothing$ and $V_{2}\cap B_{d}(x)\neq \varnothing$, then $d(V_{1}, V_{2})\leq2d<r.$ So $V_{1}$ and $V_{2}$ belong to distinct families $\mathcal{V}_{0}$ and $\mathcal{V}_{1}$.
 Therefore,
 \[
  \sharp\{V\in \mathcal{V}|V\cap B_{d}(x)\neq \varnothing \}\leq2.
  \]
\item (2)$\Rightarrow$(3): Let $\lambda>0$ be given and take a uniformly finite asymptotic dimension cover $\mathcal{V}$ of $X$
with $2\lambda$-multiplicity $\leq$ 2.
Define
$\widetilde{V}=N_{2\lambda}(V)=\{x\in X~|~d(x, V)<2\lambda\}$
and let $ \mathcal{W}=\{\widetilde{V}~|~V\in\mathcal{V}\}$. It is easy to see $ \mathcal{W}$ has uniformly finite asymptotic dimension by Lemma \ref{lem:qisomecomplexity}.  Note that $m(\mathcal{W})\leq 2$. i.e.
\[
\forall~x\in X, ~~ \sharp\{\widetilde{V}\in\mathcal{W}|~x\in\widetilde{V}\}\leq2.
\]
Indeed, if $x\in\widetilde{V}=N_{2\lambda}(V)$, then $d(x, V)<2\lambda$. It follows that
\[
V\cap B_{2\lambda}(x)\neq\varnothing.
\]
Since
\[
  \sharp\{V\in \mathcal{V}|V\cap B_{2\lambda}(x)\neq \varnothing \}\leq2,
  \]
  we have
  \[
  \sharp\{\widetilde{V}\in\mathcal{W}|~x\in\widetilde{V}\}\leq  \sharp\{V\in \mathcal{V}|V\cap B_{2\lambda}(x)\neq \varnothing \}\leq2.
  \]
 It is easy to see that $L(\mathcal{W})>\lambda$. Indeed, for every $A\subseteq X$ with diam~$A\leq\lambda$, choose $a\in A\subseteq X$.
 Since $\mathcal{V}$ is a cover of $X$, there exists $V\in\mathcal{V}$ such that $a\in V$ and hence
 \[
 A\subseteq B_{\lambda}(a)\subseteq N_{2\lambda}(V)=\widetilde{V}\in\mathcal{W}.
 \]

 \item (3)$\Rightarrow$(4): Let $\varepsilon>0$ be given and suppose that $\mathcal{W}$ is a uniformly finite asymptotic dimension cover  of $X$ with  $L(\mathcal{W})>\lambda=\frac{20}{\varepsilon}$
and $m(\mathcal{W})\leq 2$.
For each $W\in\mathcal{W}$, define $\varphi_{W}: X\longrightarrow[0,1]$ by
\[
\varphi_{W}(x)=\frac{d(x,X-W)}{\sum_{V\in\mathcal{W}}d(x, X-V)}.
\]
Note that $\varphi_{W}(x)=0$ if and only if $x~\overline{\in}~W$.
The maps $\{\varphi_{W}\}$ define a map $\varphi:X\longrightarrow K$ by
\[
\varphi(x)=\{\varphi_{W}(x)\}_{W\in\mathcal{W}}.
\]

It is easy to see that $x\in\varphi^{-1}(\text{st}~W)$, i.e. $\varphi_{W}(x)>0$
if and only if $x\in W$.
Then $\varphi^{-1}(\text{st}~W)=W$.
Hence $\{\varphi^{-1}(\text{st}~W)\}_{W\in\mathcal{W}}=\mathcal{W}$ has uniformly asymptotic dimension.
Finally, we check that $\varphi:X\longrightarrow K$ is $\varepsilon$ Lipschitz.
Note that
\[
~\abs{d(x,X-U)-d(y,X-U)}\leq d(x,y),~~\forall~x,y\in X,~\forall U\in\mathcal{W}
\]
and
\[
\sum_{W\in\mathcal{W}}d(x,X-W)\geq\frac{\lambda}{2},\forall~x\in X.
\]
Indeed, since diam~$B_{\frac{\lambda}{2}}(x)\leq\lambda$ and $L(\mathcal{W})>\lambda$,
there exists $W_{0}\in\mathcal{W}$ such that $B_{\frac{\lambda}{2}}(x)\subseteq W_{0}$.
So
\[
\sum_{W\in\mathcal{W}}d(x,X-W)\geq d(x,X-W_{0})\geq\frac{\lambda}{2}.
\]
Since  $m(\mathcal{W})\leq 2$
\[
\sharp\{W\in\mathcal{W}|~\varphi_{W}(x)>0\}\leq2,
\]
then we have

\[
\begin{split}
|\varphi_{U}(x)-\varphi_{U}(y)|&=\abs{\frac{d(x,X-U)}{\sum_{V\in\mathcal{W}}d(x, X-V)}-\frac{d(y,X-U)}{\sum_{V\in\mathcal{W}}d(y, X-V)}}\\
&\leq\abs{\frac{d(x,X-U)-d(y,X-U)}{\sum_{V\in\mathcal{W}}d(x, X-V)}}+\abs{\frac{d(y,X-U)}{\sum_{V\in\mathcal{W}}d(x, X-V)}-\frac{d(y,X-U)}{\sum_{V\in\mathcal{W}}d(y, X-V)}}\\
&\leq\frac{2}{\lambda}d(x,y)+\frac{d(y,X-U)}{\sum_{V\in\mathcal{W}}d(x, X-V)\sum_{V\in\mathcal{W}}d(y, X-V)}\sum_{V\in\mathcal{W}}\abs{d(y,X-V)-d(x,X-V)}\\
&\leq\frac{2}{\lambda}d(x,y)+\frac{2}{\lambda}*4d(x,y)\\
&=\frac{10}{\lambda}d(x,y).
\end{split}
\]
Therefore,
\[
\begin{split}
\|\varphi(x)-\varphi(y)\|_{2}&=(\sum_{U\in\mathcal{W}}\abs{\varphi_{U}(x)-\varphi_{U}(y)}^{2})^{\frac{1}{2}}\\
&\leq(4*\frac{100}{\lambda^{2}}d^{2}(x,y))^{\frac{1}{2}}\\
&=\frac{20}{\lambda}d(x,y)=\varepsilon d(x,y).
\end{split}
\]
\item (4)$\Rightarrow$(1): Let $r>0$ be given and  let $K$ be any uniform complex of dimension 1.
For each $i=0, 1$, let
    \[
    \mathcal{U}_{i}=\{st(b_{\sigma}, \beta^{2}K)~|~\sigma\subset K, dim~\sigma=i\},
    \]
    where $b_{\sigma}$ is the barycenter of $\sigma$ and $\beta^{2}K$ denotes the second barycentric  subdivision.
    It is easy to see that $ \mathcal{U}_{i}$ is $c$-disjoint for some constant $c>0$. Let $\varepsilon=\frac{c}{r}$,
    take an $\varepsilon$-Lipschitz map $\varphi: X\rightarrow K\subseteq\Delta$ to a uniform simplicial complex of dimension 1. Define
    \[
    \mathcal{V}_{i}=\{\varphi^{-1}(U)~|~U\in\mathcal{U}_{i}\},~~~~i=0,1.
    \]
    Then $\mathcal{V}_{i}$ is $r$-disjoint and $\mathcal{V}_{0}\cup\mathcal{V}_{1}$ covers $X$.
    For every $st(b_{\sigma}, \beta^{2}K)\in\mathcal{U}_{i}$, there exists $u\in K^{(0)}$ such that
    \[
    st(b_{\sigma}, \beta^{2}K)\subseteq st(u, K).
    \]
   and it follows that
      \[
    \varphi^{-1}(st(b_{\sigma}, \beta^{2}K))\subseteq \varphi^{-1}(st(u, K)).
    \]
    Since $\{\varphi^{-1}(\text{st}~u)|u\in K^{(0)}\}$ has uniformly finite asymptotic dimension,
    $$
    \mathcal{V}_{i}=\{\varphi^{-1}(st(b_{\sigma}, \beta^{2}K))~|~~\sigma\subset K, dim~\sigma=i\}
    $$
    has uniformly finite asymptotic dimension.
\item (1)$\Rightarrow$(5): Let $\mathcal{V}$ be given with diam~$\mathcal{V}\leq B$, for some positive constant $B$.
Take $r$-disjoint families $\mathcal{W}_{0}$ and $\mathcal{W}_{1}$ of uniformly finite asymptotic dimension with $r>2B$
and $\mathcal{W}_{0}\cup\mathcal{W}_{1}$ covers $X$.
For each $i=0,1$, let
\[
\mathcal{U}_{i}=\{N_{B}(W)~|~W\in\mathcal{W}_{i}\}.
\]
Since $\mathcal{W}_{i}$ is $r$-disjoint and $r>2B$, $\mathcal{U}_{i}$ is disjoint. Let $\mathcal{U}=\mathcal{U}_{0}\cup\mathcal{U}_{1}$,
then $m(\mathcal{U})\leq 2.$
For every $V\in\mathcal{V}$, since $\mathcal{W}_{0}\cup\mathcal{W}_{1}$ covers $X$, there is a $W\in\mathcal{W}_{0}\cup\mathcal{W}_{1}$ such that $W\cap V\neq\varnothing.$
Assume that $x_{0}\in W\cap V$, then for every $x\in V$, we have
\[
d(x,W)\leq d(x, x_{0})\leq \text{diam}~ V\leq \text{diam}~ \mathcal{V}\leq B.
\]
i.e. $x\in N_{B}(W)$.
Therefore, $V\subseteq N_{B}(W)\in\mathcal{U}$
and hence $\mathcal{V}$ refines $\mathcal{U}$.

\item (5)$\Rightarrow$(3): Let $\lambda>0$ be given and let $\mathcal{V}=\{B_{\lambda}(x)~|~x\in X\}$.
Clearly, $\mathcal{V}$ is a uniformly bounded cover of $X$.  So there is a cover $\mathcal{W}$ of $X$ with uniformly finite asymptotic dimension such that $\mathcal{V}$ refines $\mathcal{W}$ and $m(\mathcal{W})\leq 2$. Finally, we will show that $L(\mathcal{W})\leq \lambda$.
Indeed, for every $A\subseteq X$ and diam~$A\leq\lambda$, choose any $a\in A$, then $A\subseteq B_{\lambda}(a)\in\mathcal{V}$.
Since $\mathcal{V}$ refines $\mathcal{W}$, there is a $W\in \mathcal{W}$ such that $B_{\lambda}(a)\subseteq W$.
Therefore, $A\subseteq B_{\lambda}(a)\subseteq W$.
\end{itemize}

\end{thm}

\end{section}

\begin{section}{The exact complexity of the product of wreath products}

\begin{defi}
Let $G$ be a countable discrete group. A \emph{length function} $l : G \longrightarrow \RR_{+}$ on  $G$ is a function satisfying: for all $g, f\in G$,
\begin{itemize}
\item  $l(g)=0$ if and only if $g$ is the identity element of $G$,
\item  $l(g^{-1})=l(g),$\
\item  $l(gf)\leq l(g)+l(f).$
\end{itemize}
We say that the metric $d$, defined as follows:
\[
d(s,t)=l(s^{-1}t) ~~~\forall~s,t~\in G,
\]
is the metric induced by the length function $l$.
\end{defi}
A length function $l$ is called \emph{proper} if for all $C>0, l^{-1}([0,C])\subset G$ is finite.

Let $S$ be a finite generating set for a group $G$, for any $ g\in G$, define $|g|_{S}$ to be
the length of the shortest word representing $g$ in elements of $S\cup S^{-1}$.
We say that $|\cdot|_{S}$ is \emph{word-length function} for $G$ with respect to $S$.
The left-invariant word-metric $d_{S}$ on $G$ is induced by word-length function. i.e.,
for every $ g, h\in G$,
\[
d_{S}(g, h)=|g^{-1}h|_{S}.
\]
Note that the word-length function of a finitely generated group is a proper length function.
The Cayley graph is the graph whose vertex set is $G$, one vertex for each element in $G$
and any two vertices $ g, h\in G$ are incident with an edge if and only if $g^{-1}h\in S\cup S^{-1}$.

\begin{lem}~\rm(\cite{Willett2008})
\label{lem:groupequivalent}
A countable discrete group admits a proper length function $l$ and that any two metrics of a countable discrete group induced by proper length functions
are coarsely equivalent.
\end{lem}
By Lemma \ref{lem:qisomecomplexity}, finite decomposition complexity is a coarsely invariant property of metric spaces.
As a consequence, we say  that \emph{a discrete group has finite decomposition
complexity} if its underlying metric space has finite decomposition
complexity for some (equivalently every) metric induced by proper length function.

Let $G$ and $N$ be finitely generated groups and let $1_{G}\in G$
and $1_{N}\in N$ be their units. The \emph{support} of a function $f:
N\rightarrow G$ is the set \[
\text{supp}(f)=\{x\in N
|f(x)\neq1_{G}\}.
\]
 The direct sum $\displaystyle\bigoplus_{N}G$ of groups $G$ (or
restricted direct product) is the group of functions
\[
C_{0}(N,G)=\{f: N\rightarrow G \text{ with finite support}\}.
\]
There is a natural action of $N$ on $C_{0}(N,G)$: for all $ a\in N, x\in N, f\in C_{0}(N,G)$,
\[
 a(f)(x) =
f(xa^{-1}).
\]
The semidirect product $C_{0}(N,G)\rtimes N$ is called \emph{restricted wreath
product} and is denoted as $G\wr N$. We recall that the product
in $G\wr N$ is defined by the formula
\[
(f, a)(g, b)=(f a(g), ab) \qquad \forall f,g\in C_{0}(N,G), a,b\in N.
\]
Let $S$ and $T$ be finite generating sets for $G$ and $N$,
respectively. Let $e\in C_{0}(N,G)$ denotes the constant function
taking value $1_{G}$, and let
$
\delta^{b}_{v}:N\rightarrow G, v\in N,
b\in G
$  be the $\delta$-function, i.e.
\[
\delta^{b}_{v}(v)=b
\text{ and }\delta^{b}_{v}(x)=1_{G} \text{ for }x\neq v.
\]
 Note
that $a(\delta^{b}_{v})=\delta^{b}_{va}$ and hence
$(\delta^{b}_{v} , 1_{N})=(e, v)(\delta^{b}_{1_{N}}, 1_{N})(e, v^{-1})$.
Since every function $f\in C_{0}(N,G)$ can be presented
$\delta^{b_{1}}_{v_{1}}\cdots\delta^{b_{k}}_{v_{k}}$,
\[
(f,
1_{N})=(\delta^{b_{1}}_{v_{1}},1_{N})\cdots(\delta^{b_{k}}_{v_{k}},
1_{N}) \text{ and } (f, u)=(f, 1_{N})(e, u).
\]
 The set
$\widetilde{S}=\{(\delta^{s}_{1_{N}},1_{N}) , (e, t) |s\in S, t\in
T \}$ is a generating set for $G\wr N$.
Note that $G$ and $N$ are subgroups of $G\wr N$.

An explicit formula for the word length of wreath products was found by Parry.
\begin{lem}\rm(\cite{parry})
\label{lem:lenfun}
Let $x=(f,v)\in G \wr N$, where $f\in C_{0}(N,G)$ and $v\in N$.
Assume that $f=\delta^{b_{1}}_{v_{1}}\cdots\delta^{b_{n}}_{v_{n}}$,
let $p(x)$ be the shortest path in the Cayley graph of $N$ which
starts at $1_{N}$, visits all vertices $v_{i}$ and ends at $v$.
Then
\[
|x|_{G \wr N}=|p(x)|+\displaystyle\sum_{k=1}^{n} |b_{k}|_{G}.
\]
\end{lem}

The following statement immediately follows from the above Lemma.
\begin{coro}
\label{coro:lendirectsum}
For every $f=\delta^{b_{1}}_{v_{1}}\cdots\delta^{b_{n}}_{v_{n}}\in\bigoplus_{ \ZZ}\ZZ=C_{0}(\ZZ,\ZZ)\subset\ZZ\wr\ZZ$,
where $b_{i}\in\ZZ$ and $v_{i}\in\ZZ$. Let $w(f)$ be the shortest loop in the Cayley graph of $\ZZ$ which based at 0 and
visits all vertices $v_{i}$, then
\[
|f|_{\ZZ \wr \ZZ}=|w(f)|+\displaystyle\sum_{k=1}^{n} |b_{k}|.
\]

\end{coro}

\begin{lem}\rm(\cite{Yu2013})
\label{lem:w}
Let $G$ be a finitely generated subgroup of $GL(n,\RR)$ for some natural number $n$, then $G\in \mathcal{D}_{\omega}$.

\end{lem}

\begin{thm}
\label{thm:productwreath}
For every $m\in\NN $,
let $G=(\ZZ \wr \ZZ)^{m}=\underbrace{\ZZ \wr \ZZ\times
\ZZ \wr \ZZ\times\cdots\times \ZZ \wr \ZZ}_{m}$.
Then $G \in
\mathcal{D}_\omega$ and for any $ \alpha < \omega, G
~\bar{\in}~ \mathcal{D}_{\alpha}$. i.e. the exact complexity of $G$ is $\omega$.

\begin{proof}
Since $\ZZ \wr \ZZ$ is a finitely generated group, $(\ZZ \wr \ZZ)^{m}$ is finitely generated.
By Lemma \ref{lem:w}, it suffices to show that $(\ZZ \wr \ZZ)^{m}$ is a subgroup of $GL(n,\RR)$ for some natural number $n$.
Define a map
$
\psi: \ZZ \wr \ZZ\rightarrow GL(2,\RR)
$ as follows: $\forall x=(f,n)\in\ZZ \wr \ZZ$, where $f\in\bigoplus_{\ZZ}\ZZ, n\in\ZZ$,

\begin{equation*}
\begin{split}
\psi(x) = \begin{pmatrix}{}1 & 0\\ \sum_{k\in \text{supp}f}f(k)\pi^{k} & \pi^{n}\end{pmatrix}
\end{split}
\end{equation*}
Now we will show that \[\forall x=(f,n), y=(g,m)\in\ZZ \wr \ZZ,~~ \psi(xy)=\psi(x)\psi(y).\]
Indeed,

\begin{equation*}
\begin{split}
\psi(x)\psi(y)& =
 \begin{pmatrix}{}1 & 0\\ \sum_{k\in \text{supp}f}f(k)\pi^{k} & \pi^{n}\end{pmatrix}\begin{pmatrix}{}1 & 0\\ \sum_{k\in \text{supp}g}g(k)\pi^{k} & \pi^{m}\end{pmatrix}\\&=\begin{pmatrix}{}1 & 0\\ \sum_{k\in \text{supp}f}f(k)\pi^{k}+\sum_{k\in \text{supp}g}g(k)\pi^{k+n} & \pi^{n+m}\end{pmatrix}
\end{split}
\end{equation*}

Since $xy=(f,n)(g,m)=(fn(g),n+m)$ and $fn(g)(k)=f(k)+g(k-n)$,
\begin{equation*}
\begin{split}
\psi(xy)& =
 \begin{pmatrix}{}1 & 0\\ \sum_{k\in \text{supp}f\cup\text{supp}n(g)}[f(k)+g(k-n)]\pi^{k} & \pi^{n+m}\end{pmatrix}
\end{split}
\end{equation*}

Note that
\[
\sum_{k\in \text{supp}f}f(k)\pi^{k}+\sum_{k\in \text{supp}g}g(k)\pi^{k+n}=\sum_{k\in \text{supp}f}f(k)\pi^{k}+\sum_{k\in \text{supp}n(g)}n(g)(k)\pi^{k}
= \sum_{k\in \text{supp}f\cup\text{supp}n(g)}[f(k)+g(k-n)]\pi^{k}.
\]

It follows that
\[
\forall x=(f,n), y=(g,m)\in\ZZ \wr \ZZ,~~ \psi(xy)=\psi(x)\psi(y).
\]
Define a map $\widetilde{\psi}:(\ZZ \wr \ZZ)^{m}\longrightarrow GL(2m, \RR)$ as follows: $\forall x=(x_{1},x_{2},\cdots, x_{m})\in(\ZZ \wr \ZZ)^{m}$,
\begin{equation*}
\begin{split}
\widetilde{\psi}(x) =
\text{diag}
( \psi(x_{1}),  \psi(x_{2}),  \psi(x_{3}), \ldots,  \psi(x_{m}))\in GL(2m,\RR).
\end{split}
\end{equation*}
It is easy to check that $\widetilde{\psi}$ is a group homomorphism. So $(\ZZ \wr \ZZ)^{m}$ can be considered as a subgroup of $GL(2m,\RR)$.

 Finally, since $\ZZ \wr \ZZ$ is a subgroup of
$(\ZZ \wr \ZZ)^{m}$ and $\ZZ \wr \ZZ$ does not have asymptotic dimension, we have
\[
\forall~\alpha < \omega, G
~\bar{\in}~ \mathcal{D}_{\alpha}.
\]
Therefore, the exact complexity of $G$ is $\omega$.

\end{proof}

\end{thm}

 \begin{lem}{\rm(\cite{Yan2011})}
\label{lem:ZZ}
Let $H $ be a countable group and $H^{m}=\underbrace{H\times
H\times\cdots\times H}_{m}$. For every $ r\in\NN, \text{ there
exist } m\in\NN $ and a metric family $\mathcal{Y}$ such that
\begin{itemize}
 \item[(1)] $H \wr \ZZ\stackrel{r}{\rightarrow}\mathcal{Y}, $
 \item [(2)]
there is a coarse embedding from $\mathcal{Y}$ to $\{g H^{m}\}_{g
\in \bigoplus H}.$
\end{itemize}

\end{lem}

\begin{thm}
\label{thm:multiwreath}

$(\ZZ \wr \ZZ)\wr\ZZ\in
\mathcal{D}_{\omega+1}$.
i.e. the  complexity of $(\ZZ \wr \ZZ)\wr\ZZ$ is $\omega+1$.
\end{thm}

\begin{proof}
Let $H=\ZZ\wr\ZZ$, by Lemma \ref{lem:ZZ}, for every $ r\in\NN, \text{ there
exist } m\in\NN $ and a metric family $\mathcal{Y}$ such that
\begin{itemize}
 \item[(1)] $H \wr \ZZ\stackrel{r}{\rightarrow}\mathcal{Y}, $
 \item [(2)]
there is a coarse embedding from $\mathcal{Y}$ to $\{g H^{m}\}_{g
\in \bigoplus H}.$
\end{itemize}
By Theorem \ref{thm:productwreath}, $H^{m}\in\mathcal{D}_{\omega}$.
Then it is easy to obtain that $\{g H^{m}\}_{g
\in \bigoplus H}\in\mathcal{D}_{\omega}$. Since there is a coarse embedding from $\mathcal{Y}$ to $\{g H^{m}\}_{g
\in \bigoplus H}$, $\mathcal{Y}\in\mathcal{D}_{\omega}$ by Lemma \ref{lem:qisomecomplexity}.
Therefore, $(\ZZ \wr \ZZ)\wr\ZZ=H\wr\ZZ\in
\mathcal{D}_{\omega+1}$.

\end{proof}

{\bf Acknowledgments.}  We thank Jiawen Zhang for many useful discussions.

\end{section}


\begin{thebibliography}{10}


\bibitem{Gromov}
M. Gromov, \emph{Asymptotic invariants of infinite groups.} in: Geometric Group Theory, Vol.2, Sussex, 1991, in: Lond. Math. Soc. Lect. Note Ser., vol.182, Cambridge Univ. Press, Cambridge, 1993, pp.1--295.


\bibitem{Yu2012}
E. Guentner, R. Tessera, G. Yu, \emph{A notion of geometric complexity and its application to topological rigidity.} Invent. Math. 189(2012), 315--357.

\bibitem{Yu2013}
E. Guentner, R. Tessera, G. Yu, \emph{Discrete groups with finite decomposition complexity.} Groups Geom. Dyn. 7 (2013), 377--402.

\bibitem{cannon1996}
J. W. Cannon, W. J. Floyd and W. R. Parry, \emph{Introductory notes on Richard Thompson¡¯s groups.} Enseign.
Math. (2) 42 (1996), no. 3-4, 215--256.


\bibitem{Yan2014}
Yan Wu, Xiaoman Chen, \emph{Distortion of Wreath Products in Thompson's Group F.} Chinese Annals of Mathematics,Series B, 35B(5) (2014), 801--816.

\bibitem{Yan2011}
Yan Wu, Xiaoman Chen,\emph{ On Finite Decomposition Complexity of Thompson
Group.} Journal of Functional Analysis, Vol.261, Issue 4, (2011), 981--998.


\bibitem{Bell2011}
G. Bell, A. Dranishnikov, \emph{Asymptotic dimension in Bedlewo.} Topology
Proc. 38 (2011), 209--236.


\bibitem{parry}
W.~Parry, \emph{Growth series of some wreath products.} Trans. Am. Math. Soc. 331(1992), no. 2, 751--759.









\bibitem{Willett2008}
R. Willett, \emph{Some notes on property A}, Limits of graphs in group theory and computer science,
  191--281, 2009.


\end{thebibliography}

\providecommand{\bysame}{\leavevmode\hbox to3em{\hrulefill}\thinspace}
\providecommand{\MR}{\relax\ifhmode\unskip\space\fi MR }
\providecommand{\MRhref}[2]{%
  \href{http://www.ams.org/mathscinet-getitem?mr=#1}{#2}
}
\providecommand{\href}[2]{#2}

\end{document}